\documentclass[twoside,reqno]{amsart}
\usepackage{graphicx}
\usepackage{amsmath}
\usepackage{amssymb}
\usepackage{mathrsfs}
\usepackage[mathcal]{eucal}
\usepackage{color}

\usepackage{hyperref}
\usepackage[alphabetic, initials]{amsrefs}
\DeclareMathOperator{\lcm}{lcm}

\begin{document}

\title[Formal factorization of irregular operators]{Formal factorization of higher order irregular linear differential operators}

\author[L. Mezuman]{Leanne Mezuman}
\author[S. Yakovenko]{Sergei Yakovenko$^\dag$}
\address{
Department of Mathematics\\
Weizmann Institute of Science,
Israel}
\email{{\{leanne.mezuman, sergei.yakovenko\}@weizmann.ac.il}}
\thanks{$^\dag$Gershon Kekst Professor of Mathematics}

\begin{abstract}
We study the problem of decomposition (non-commutative factorization) of linear ordinary differential operators near an irregular singular point. The solution (given in terms of the Newton diagram and the respective characteristic numbers) is known for quite some time, though the proofs are rather involved. 

We suggest a process of reduction of the non-commutative problem to its commutative analog, the problem of factorization of pseudopolynomials, which is known since Newton invented his method of rotating ruler. It turns out that there is an ``automatic translation'' which allows to obtain the results for formal factorization in the Weyl algebra from well known results in local analytic geometry.

Besides, we draw some (apparently unnoticed) parallelism between the formal factorization of linear operators and formal diagonalization of systems of linear first order differential equations.

\end{abstract}

\date{May 10, 2018}
\let\parasymbol=\S
\def\secref#1{\parasymbol\ref{#1}}
\def\C{{\mathbb C}}
\def\Cs{{\mathscr C}}
\def\N{{\mathbb N}}
\def\Q{{\mathbb Q}}
\edef\Ore{\O}
\def\O{{\mathscr O}}
\def\M{{\mathscr M}}
\def\W{{\hat{\mathscr W}}}
\def\G{{\varGamma}}
\def\R{{\mathbb R}}
\def\F{{\mathscr F}}
\def\Z{{\mathbb Z}}
\def\P{{\mathbb P}}
\def\E{{\scriptstyle{\mathscr E}}}
\def\H{\boldsymbol H}
\def\Ker{\operatorname{Ker}}
\def\:{\colon}
\def\eu{{\epsilon}}
\def\f{\varphi}
\def\l{\lambda}
\def\L{\varLambda}
\def\ord{\operatorname{ord}}
\def\Mat{\operatorname{Mat}}
\def\^{\hat}
\def\~{\widetilde}
\def\ssm{\smallsetminus}
\def\le{\leqslant}
\def\ge{\geqslant}
\def\GL{{\operatorname{GL}}}
\def\id{\operatorname{id}}
\edef\iorig{\i}
\def\i{\mathrm i}
\def\d{\partial}
\def\<{\left<}
\def\>{\right>}
\def\S{\varSigma}
\def\lcm{\operatorname{lcm}}
\def\sh#1{^{\scriptscriptstyle[#1]}}
\def\supp{\operatorname{supp}}
\def\D{\Delta}
\def\Sym{\operatorname{Sym}}
\def\Dif{\operatorname{Dif}}
\def\Lead{\operatorname{Lead}}
\def\s{\operatorname{\boldsymbol\varsigma}}
\def\e{\varepsilon}

\let\chiorig=\chi
\def\chi{\raise 1.5pt\hbox{$\chiorig$}}
\def\wgt{\operatorname{wgt}}
\def\PS{\operatorname{\mathscr S}}
\def\PP{{\mathscr P}}
\def\U{{\mathscr U}}

\newtheorem{Thm}{Theorem}
\newtheorem{Lem}{Lemma}
\newtheorem{Prop}{Proposition}
\newtheorem{Cor}{Corollary}

\theoremstyle{definition}
\newtheorem{Def}{Definition}
\newtheorem{Prob}{\textcolor{red}{Problem}}
\newtheorem{Ex}{Example}

\theoremstyle{remark}
\newtheorem{Rem}{Remark}

\maketitle

\section{Introduction}

The local theory of linear ordinary differential equations exists in two closely related but different flavors. First, one can consider systems of first order linear equations near a singular point. Such systems form an infinite-dimensional space on which several groups of gauge transformations act naturally. Then the classification problem arises: what is the simplest normal form to which a given system can be reduced by a gauge transformations. This theory is well developed, in particular, delicate results explaining the difference between formal and convergent classification (the Stokes phenomenon) were obtained half a century ago.

Another flavor of the theory deals with (scalar) higher order linear differential equations involving only one unknown function. Formally such equations can be reduced to systems of first order equations and vice versa, but the natural group action is lost by such reduction. Instead a notion of Weyl equivalence can be introduced, which makes the classification problem meaningful once again.

The two theories are closely parallel (but clearly different) for the mildest type of singularities, the Fuchsian (regular) ones, as was shown in \cite{shira}. In this paper we discuss the theorem by Bernard Malgrange  \cite{malgrange} which is an analogue of a theorem of formal diagonalization of non-resonant irregular singularities \cite{thebook}*{Theorem 20.7}.

We start with a brief summary of the theory of systems of first order equations. For simplicity from the very beginning we concentrate on the formal case, leaving the issue of convergence for remarks.

\subsection{Systems of first order linear ordinary differential equations}
Denote by $\C[[t]]$ the differential ring of formal Taylor series and $\Bbbk=\C[t^{-1}][[t]]$ its quotient differential field of Laurent polynomials with the usual derivation $\d=\frac{\mathrm d}{\mathrm dt}$.  A \emph{system of first order linear ordinary differential equations} over $\Bbbk$ is defined by an $n\times n$ matrix $M=\{M_{ij}\}\in\Mat(n,\Bbbk)$ and has the form
\begin{equation*}
  \tfrac{\mathrm d}{\mathrm dt} x_i=\sum _{j=1}^n M_{ij}(t)x_j,\qquad i=1\dots,k.
\end{equation*}
It is more convenient to write this equation in the matrix form with respect to the unknown $n\times n$-matrix function $X$, specifically singling out the order of the pole of the coefficients matrix as follows,
\begin{equation}\label{ls-2}
  t^{1+r}\,\tfrac{\mathrm d}{\mathrm dt} X=A(t)X,\quad r\in\Z_+,\  A=A_0+A_1t+A_2t^2+\cdots\in\Mat(n,\C[[t]]).
\end{equation}
The integer $r\ge 0$ is called the \emph{Poincar\'e index} of the system \eqref{ls-2}; if $r=0$, the system is called \emph{Fuchsian}, the leading matrix $A_0$ of the matrix formal Taylor series $A(t)$ is assumed nonzero. 

The group of \emph{formal gauge transformations} $\GL(n,\C[[t]])$ acts naturally on linear systems of the form \eqref{ls-2} by ``change of variables'': if $H(t)=H_0+H_1t+H_2t^2+\cdots$, $\det H_0\ne 0$ is a formal matrix series, then the transformed system for the new ``unknown'' matrix $Y=H(t)X$ takes the form $t^{1+r}\,\tfrac{\mathrm d}{\mathrm dt} Y=B(t)Y$ with the new matrix coefficient $B(t)=t^{r+1}(\tfrac{\mathrm d}{\mathrm dt}H)H^{-1}-HA(t)H^{-1}$.

The natural question is to describe the orbits of this action, in particular, determine what is the ``simplest'' form to which a given system can be reduced by a suitable formal gauge transformation. This question is almost completely settled for Fuchsian systems, including the issue of convergence for holomorphic systems and holomorphic gauge transformations. The question for non-Fuchsian systems is much more subtle, especially the issue of convergence, yet the first step of the formal classification is rather simple.

\begin{Def}
An non-Fuchsian system \eqref{ls-2} is resonant, if among the eigenvalues $\l_1,\dots,\l_n$ of the leading matrix $A_0$ occur equal numbers with zero differnce. Otherwise (when all eigenvalues are pairwise different) the system is non-resonant, see \cite{thebook}*{\parasymbol 20C}.
\end{Def}

\begin{Rem}
For Fuchsian systems the resonance condition means that some of the eigenvalues differ by a natural number, i.e., $\l_i-\l_j\in\N$ for some $i\ne j$.
\end{Rem}

\begin{Thm}\label{thm:diag}
A non-resonant non-Fuchsian system can be formally diagonalized, i.e., there exists a formal gauge transformation such that the corresponding transform $B(t)$ becomes a diagonal matrix.
\end{Thm}

In other words, in the non-resonant case the system can be decomposed into a Cartesian product of one-dimensional equations. Appearance of resonances (multiple eigenvalues of $A_0$) leads to a more involved formal normal form. The analytic reasons for the divergence (in general) of the diagonalizing gauge transform (the Stokes phenomenon) are also well understood, see \cite{thebook}*{\parasymbol 16 and \parasymbol 20}.

\subsection{Higher order linear operators}\label{sec:operators}
The other flavor of the theory deals with linear equations involving only one unknown (scalar) function $u$, but several derivatives. For simplicity we will consider only the \emph{homogeneous} equations of this type, which can always be written under the form
\begin{equation}\label{lode}
 a_0(t)u^{(n)}+a_1(t)u^{(n-1)}+\cdots+a_{n-2}(t)u''+a_{n-1}(t)u'+a_n(t)u=0,
\end{equation}
where $a_0,\dots,a_n\in\Bbbk$ are the coefficients, defined modulo a multiplication by a nonzero Laurent series from $\Bbbk$. In particular, one can assume that the leading coefficient $a_0$ is identically one, or on the contrary, assume that all $a_0,\dots,a_n$ are formal Taylor series (not involving negative powers of $t$). However, it turns out that the smart decision, simplifying many formulations is to use a different derivation when expanding a linear dependence between derivatives of the unknown function.

The equation \eqref{lode} can be rewritten in the operator form. Denote by $\d\:\Bbbk\to\Bbbk$, $\d\:u\mapsto\tfrac{\mathrm d}{\mathrm dt}$ the standard derivation, and identify each element $a\in\Bbbk$ with a ``zero order operator'' $u\mapsto au$. Then the left hand side of \eqref{lode} can be interpreted as the result of application of the differential operator
\begin{equation}\label{lodo}
  L=\sum_{j=0}^n a_j\d^{n-j},\qquad a_0,\dots, a_n\in\Bbbk.
\end{equation}
to the unknown function $u$ (which may well be in any extension of the field $\Bbbk$). The Leibniz rule implies the commutation law
\begin{equation}\label{leib}
  \d t^k=kt^{k-1}\d+t^k,\qquad k\in\Z.
\end{equation}

Denote by $\eu$ the Euler derivation,
\begin{equation}\label{euler}
  \eu=t\,\d=t\,\tfrac{\mathrm d}{\mathrm dt}.
\end{equation}
Then any linear operator of the form \eqref{lodo} can be re-expanded as the sum
\begin{equation}\label{lodo-1}
  L=\sum_{j=0}^n b_j(t)\eu^{n-j}, \qquad b_j\in\Bbbk,
\end{equation}
where the coefficients $b_j\in\Bbbk$ as before are defined modulo a nonzero element in $\Bbbk$. The $\C$-linear space of such operators will be denored $\Bbbk[\eu]$. It is a non-commutative algebra with respect to the operation of composition. The commutation law in $\Bbbk[\eu]$ is given by the formula
\begin{equation}\label{comm}
  \eu^j t^k=t^k(\eu+k)^j, \qquad t\in\Z,\ j\in\Z_+,
\end{equation}
which looks especially simple compared with the law \eqref{leib} extended on arbitrary monomials.

\begin{Def}
A \emph{canonical representation} of an $n$th order linear differential operator is the representation \eqref{lodo-1} in which:
\begin{itemize}
 \item all coefficients $b_0,\dots,b_n\in\C[[t]]$ are formal Taylor polynomials, not involving negative powers of $t$, and at least one of them has a nonzero free term, $b_j(0)\ne0$;
 \item all coefficients $b_0,\dots,b_n$ appear to the left from the symbols of the iterated Euler derivations $\eu^j$.
\end{itemize}
\end{Def}


Alternatively, any operator in the canonical representation can expanded as an infinite series of the form
\begin{equation}\label{oper-series}
  L=\sum_{j\ge 0}t^j p_j(\eu),\qquad p_j\in\C[\eu],\ \max_j\deg_\eu p_j=n=\ord L.
\end{equation}

\begin{Def}\label{def:Fuchsian}
An operator is Fuchsian, or \emph{regular}, if in any canonical representation the leading coefficient is nonvaninishing, $b_0(0)\ne0$. Otherwise it is called \emph{irregular}.

The expansion \eqref{oper-series} corresponds to a Fuchsian operator, if and only if $\deg_\eu p_0=n=\ord L$.
\end{Def}

\begin{Rem}\label{rem:reduction}
A Fuchsian equation \eqref{lodo-1} can be reduced to a Fuchsian system \eqref{ls-2} in the standard way by introducing the formal variables $x_j=\eu^{j-1}u$, $j=1,\dots,n-1$. The condition of regularity is defined for equations with meromorphic coefficients in terms of the growth rate of their solutions. For scalar equations regularity is equivalent to Fuchsianity.

Conversely, a Fuchsian system can be written in the (matrix) operator form as $(\eu-A)X=0$ which \emph{mutatis mutandis} is Fuchsian in the sense of the above Definition.
\end{Rem}

\subsection{Weyl equivalence}
The (infinite-dimensional $\C$-linear) space of differential operators admits no natural action of a gauge transformations group that would be large enough (changes of variable of the form $v=h(t)u$ with a formal series $h\in\C[[t]]$ are obviously insufficient for a meaningful classification). Instead one can use the fact that differential operators form a \emph{noncommutative algebra}. The following definition was suggested in \cite{shira} based on the fundamental work by \Ore.~Ore \cite{ore}.

\begin{Def}
Two linear ordinary differential operators $L,M\in\Bbbk[\eu]$ are \emph{Weyl equivalent}, if there exist two \emph{Fuchsian} operators $H,K\in\Bbbk[\eu]$ such that:
\begin{enumerate}
 \item $MH=KL$, and
 \item $\gcd (H,L)=1$, i.e., there is no nontrivial operator $P\in\Bbbk[\eu]$ such that both $H$ and $L$ are divisible (from the right) by $P$.
\end{enumerate}
\end{Def}
Informally, two operators are Weyl equivalent, if there exists a Fuchsian operator $u\mapsto v=Hu$ which maps in a bijective way solutions of the equation $Lu=0$ to solutions of the equation $Mv=0$. It is not obvious why this is indeed an equivalence relation (in particular, why it is symmetric), yet this can be verified \cite{ore,shira}.

It turns out that the Weyl classification of \emph{Fuchsian} operators is very much similar to that of the gauge equivalence of the Fuchsian systems of linear equations. In particular, see \cite{shira}:
\begin{itemize}
 \item in the generic (non-resonant) case a Fuchsian operator is Weyl equivalent to an Euler operator from $\C[\eu]$ (i.e., with constant coefficients);
 \item in the resonant case the normal form is a composition of \emph{polynomial} first order operators $\eu-\l_j(t)$, $\l_j\in\C[t]$, with the degrees of the polynomials $\l_j$ depending on the combinatorial structure of resonances;
 \item the normal form is Liouville integrable;
 \item for a Fuchsian operator $L$ with holomorphic (convergent) coefficients, the normal form and the conjugating operators $H,L$  also have holomorphic coefficients.
\end{itemize}

\subsection{Non-commutative factorization}
The mere possibility of non-commutative factorization of Fuchsian operators into terms of first order is a simple fact. It suffices to note that any Fuchsian equation always has a solution of the form $u(t)=t^\l v(t)$ with $\l\in\C$ and an invertible series $v\in\C[[t]]$ (in the analytic category this solution is generated by an eigenvector of the monodromy operator). Such solution immediately produces a right Fuchsian factor of order 1 for the corresponding operator. The  difficult part of \cite{shira} is to reduce \emph{simultaneously} all factors to polynomial forms by a suitable Weyl equivalence.

In this paper we discuss a much simpler question on the \emph{possibility} of the noncommutative factorization of irregular differential operators. To the best of our knowledge, this problem was first addressed by B.~Malgrange, who in 1979 sketched a solution in a preprint published only in 2008 \cite{malgrange}. Soon a different proof based on the valuations theory was published by P. Robba \cite{robba}. In both cases the answer was given in terms of the Newton diagram of the differential operator, yet the proof was essentially noncommutative.

An analogous question in the commutative \emph{algebra of pseudopolynomials} $\C[[t]][\xi]$ was first studied by I.~Newton in 1676 in a letter to H.~Oldenburg that was published only in 1960 according to \cite{arnold,brieskorn}. Newton invented his method of a rotating ruler which today is formalized using the Newton polygon (resp., Newton diagram) to solve this problem.

Even in the commutative case the Newton's solution was considerably involved, see \cite{vain-tren} for the modern exposition; an appropriate modification of this proof allows to treat also the noncommutative case of differential operators, see the excellent textbook \cite{vdp-sing}. However, the modern techniques of the singularity theory (blow-up) allow to obtain the same results in a much simpler way.

In our paper we develop a formal technique which allows to transfer \emph{all} results for commutative pseudopolynomials to the noncommutative case of differential operators and outline the similarity between the respective results. In particular, we prove a result that is a direct analogue of Theorem~\ref{thm:diag} (the necessary definitions are introduced below).

\begin{Thm}
Let $L$ be a single-slope differential operator $L$ with the rational slope $r=p/q$, $\gcd(p,q)=1$, and the roots $\l_1,\dots,\l_m\in\C$ of the corresponding characteristic polynomial are nonresonant, i.e., pairwise different.

Then the operator $L$ can be formally decomposed as a noncommutative product of $m$ irreducible operators, $L=L_1\cdots L_m$, with the same slope $r$ having the form $L_j=t^p\eu^q-\l_j+\cdots$.
\end{Thm}

Note that if the slope is integer, then $q=1$ and the irreducible factors are of order $1$. 

The general factorization statement is given in Theorem~\ref{thm:main} below: its structure is completely analogous to the structure of the classical factorization theorem for pseudopolynomials. We start with a brief recap of the commutative theory in the form most suitable for our purposes.

\subsection{Acknowledgements}
This paper appeared after a thorough rethinking of the thesis of the first author \cite{leanne}. We are grateful to many friends and colleagues who came out with most helpful remarks after hearing conference presentations of the results, especially Jeanne-Pierre Ramis, Michael Singer, Daniel Bertrand, Gal Binyamini and Dmitry Novikov. The second author is incumbent of the Gershon Kekst Chair of Mathematics. 

\section{Pseudopolynomials and their factorization}

\subsection{Pseudopolynomials}
A \emph{pseudopolynomial} is a family of polynomials of degree $\le n=\deg P$ which formally depends on a local parameter $t\in(\C,0)$. The space of pseudopolynomials can be naturally identified with the commutative algebra that in a sense cross-bread between the algebra of polynomials and the algebra of formal Taylor series,
\begin{equation}\label{comm-alg}
\^\Cs=\C[\xi]\otimes_\C\C[[t]].
\end{equation}
Each element of this algebra can be expanded into the formal series
\begin{equation}\label{taylor}
 P(t,\xi)=\sum_{j=0}^\infty t^j p_j(\xi),\qquad p_j\in\C[\xi],\quad \deg P=\sup_{j}\deg p_j<+\infty
\end{equation}
(note the boundedness of $\deg p_j$) or as a formal double sum
\begin{equation}\label{double}
 P(t,\xi)=\sum_{(i,j)\in S}c_{ij}t^j\xi^i,\qquad S\subset \Z^2_+
\end{equation}
The set $S\subset\Z^2_+$ which belongs to the vertical strip $0\le i\le n=\deg P$ is called the \emph{support} of the pseudopolynomial $P$ and denoted by $\supp P$.

\begin{Def}\label{def:NPC}
The Newton polygon $\D_P\subseteq\R_+^2$ of a pseudopolynomial $P\in\^\Cs$ is the minimal closed convex set containing the origin $(0,0)$ and the support $\supp P$, which is invariant by the vertical translation $(i,j)\mapsto (i,j+1)$.
\end{Def}

One can immediately see that the boundary of any Newton polygon consists of two vertical rays over the points $i=0$ and $i=n$ and the graph of a convex piecewise linear function  $\chi_P\:[0,n]\to\R_+$, called the \emph{gap function}. This function is non-decreasing, which implies the following obvious conclusion.

\begin{figure}
  \centering
  \includegraphics[width=0.4\hsize]{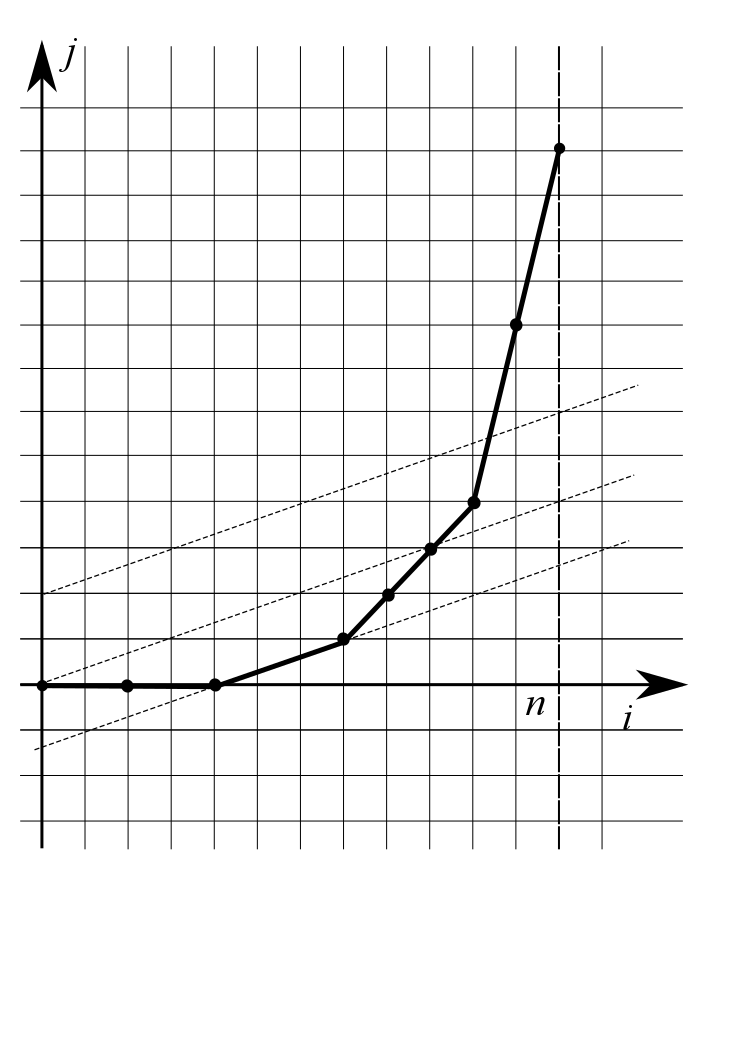}\\
  \caption{Newton diagram of a pseudopolynomial}\label{fig:ND}
\end{figure}

\begin{Prop}\label{prop:left}
If $(i,j)\in \D_P$ and $0\le i'<j$, then $(i',j)\in\D_P$. \qed
\end{Prop}

\begin{Rem}\label{rem:local}
This definition is an obvious modification of the standard notion of the Newton polygon $\D_f$ for Taylor series $f$ from $\C[[x,y]]$, defined as the minimal closed convex set which contains the support $\supp f$ and is invariant by the shifts $(i,j)\mapsto (i,j+1)$ and $(i,j)\mapsto (i+1,j)$, see \cite{wall}. It suffices to notice that for a pseudopolynomial $P(t,\xi)$ of degree $n$ the  Laurent series $f(x,y)=y^n P(x,\frac 1y)$ does not involve negative powers of $y$. The (usual) Newton polygon for $f$ is obtained by reflection of $\D_P$ in the horizontal axis and shift upwards by $n$.
\end{Rem}

The following properties of the gap function immediately follow from its construction:
\begin{enumerate}
 \item $\chi$ is defined on $[0,n]$ and $\chi(0)=0$;
 \item $\chi$ is convex, monotone non-decreasing and piecewise-linear (more accurately, piecewise-affine);
 \item $\chi$ may be non-differentiable at a point $i\in[0,1]$ if and only if $i$ and $\chi(i)$ are both integer numbers (in which case we call $(i,j)\in\D_P$ the called the \emph{corner point} or the \emph{vertex} of $\D_P$).
\end{enumerate}

\begin{Rem}
The inverse to the gap function is the smallest concave majorant of the \emph{degree function} $j\mapsto \deg p_j$ derived from the expansion \eqref{taylor}.
\end{Rem}

\begin{Def}
The union of all finite edges of the Newton polygon $\D_P$ is called the \emph{Newton diagram} of the pseudopolynomial $P$ and denoted by $\G_P$. Thus the Newton polygon is the epigraph of the gap function.
\end{Def}

\begin{Def}\label{def:admissible}
A closed convex polygon $\D\subset\R^2_+$ which is the epigraph of a convex piecewise-linear function $\chi=\chi_\D$ as above, is called \emph{admissible}. The function $\chi=\chi_\D$ will be called the \emph{gap function} for $\D$.

Collection of different slopes (derivatives) of the affine pieces of the function $\chi$, all of them nonnegative rational numbers, will be called the \emph{Poincar\'e spectrum} $\PS(\D)\subset\Q_+$ of the polygon $\D$.
\end{Def}

\begin{Def}
We call a pseudopolynomial $P$ (resp., its Newton polygon $\D$) a \emph{single-slope} pseudopolynomial (resp., polygon), if its Poincar\'e spectrum consists of a single value, $\PS(P)=\{\rho\}$. The corresponding gap function is linear on some segment $[0,d]$, $\chi_P(i)=\rho i$, $\rho\in\Q_+$. The value $\rho=0$ is not excluded.
\end{Def}

\begin{Ex}\label{ex:FuchsianPP}
By this definition $\PS(P)=\{0\}$ if and only if $\deg P=\max_j \deg p_j=\deg p_0$. We call such (single-slope) pseudopolynomials \emph{Fuchsian}, cf.~with Definition~\ref{def:Fuchsian}.
\end{Ex}

\begin{Rem}
The reason why the collection of slopes is referred to as the Poincar\'e spectrum, is as follows. A linear system \eqref{ls-2} of Poincar\'e rank $r\in\Z_+$ can be written as $t^r\eu X=A(t)X$, and after reduction to a scalar equation as explained in Remark~\ref{rem:reduction}, it will generically produce a single-slope operator with the integer slope $r$.
\end{Rem}

\subsection{Newton polygon of a product}
The key property of the Newton polygon is its ``logarithmic behavior'' with respect to multiplication in $\^\Cs$, which generalizes the geometry of superscripts in the identity $\xi^n\xi^m=\xi^{n+m}$.

\begin{Prop}\label{prop:minksum}
For any $P,Q\in\^\Cs$,
\begin{equation}\label{NPlog}
 \D_{PQ}=\D_P+\D_Q,
\end{equation}
where the right hand side is the Minkowsky sum $\{u+v\:u\in \D_P,\ v\in \D_Q\}$. \qed
\end{Prop}
For monomials this follows from the identity for their (one-point) supports, $\supp (t^{i+i'}\xi^{j+j'})=\supp (t^i\xi^j)+\supp (t^{i'}\xi^{j'})$. This immediately implies the inclusions
\begin{equation}\label{additive}
  \supp(PQ)\subseteq\supp (P)+\supp (Q),\qquad\text{hence}\qquad \D_{PQ}\subseteq\D_P+\D_Q.
\end{equation}
The inclusion for supports can be strict, since a lattice point from $\supp (P)+\supp (Q)$ can be represented in several possible ways as the sum of points from $\supp(P)$ and $\supp(Q)$. A cancellation of different contributions is possible so that the corresponding coefficient of $PQ$ could be zero. The not-so-obvious claim is that the coefficients corresponding to the \emph{corner} points of $\D_P+\D_Q$ cannot vanish because of such cancellation.

\begin{Cor}
$$
 \PS(P+Q)=\PS(P)\cup \PS(Q).\qed
$$
\end{Cor}

As follows from the Proposition~\ref{prop:minksum}, the problem of factorization of a pseudopolynomial $R\in\^\Cs$ reduces (although is \emph{not equivalent}) to the problem of representing the Newton polynomial $\D=\D_R$ as the Minkowski sum of two \emph{admissible} polynomials, $\D=\D'+\D''$. The admissibility constraints (nonnegativity, vertices only at the integer points of the lattice, two vertical bounding rays etc.) imply the following two geometrically rather obvious statements.

\begin{Lem}
An admissible \textup(in the sense of Definition~\ref{def:admissible}\textup) polygon is \emph{indecomposable}, i.e., cannot be represented as a Minkowski sum of two nontrivial admissible polygons, if and only if it has a single slope and the non-vertical edge carries no lattice points of $\Z^2_+$.

Any admissible polygon can be decomposed into the Minkowski sum of the indecomposable polygons.
\end{Lem}

\begin{proof}
It can be immediately verified that any admissible polygon can be decomposed into the Minkowski sum of the single-slope polygons. The claim of (in)decomposability for the single-slope polygons is essentially a one-dimensional statement about lattice segments in $\Z^1$.
\end{proof}

\subsection{Quasihomogeneous pseudopolynomials}\label{sec:qhg}
Let $w\in\Q_+$ be a nonnegative rational number and $\wgt=\wgt_w\:\Z^2\to\Q$ the weight function which associates to a monomial $t^j\xi^i$ the weight $\wgt(t^j\xi^i)=\wgt(i,j)=j-wi$.

\begin{Def}
A pseudopolynomial is called $w$-quasihomogeneous, if all its monomials have the same $w$-weight $\alpha\in\Q$,
$$P_\alpha=\sum_{(i,j)\:\wgt_w(i,j)=\alpha}c_{ij}t^j\xi^i.$$
One can instantly see that support of a quasihomogeneous polynomial belongs to a line with the slope $w$ and is finite (i.e., $P_\alpha\in\C[t,\xi]$ is a genuine polynomial).
\end{Def}

A quasihomogeneous polynomial of \emph{weight zero} is essentially a polynomial of a single variable. If $w=p/q$ is an irreducible fraction, then all monomials of weight zero are necessarily powers of the generating monomial $t^q\xi^p$ of weight zero, thus $P_0(t,\xi)=\sigma(t^q\xi^p)$ for some $\sigma=\sigma_P\in\C[\l]$. It is always reducible: if $\l_1,\dots,\l_k$ are the complex roots of $\sigma$, then
$$
P_0(t,\xi)=c\prod_{s=1}^k (\l_s-t^q\xi^p),\qquad c\in\C,\ c\ne0,\quad \sigma_P(\l_s)=0.
$$
An arbitrary quasihomogeneous (pseudo)polynomial $P_\alpha$ of weight $\alpha$ can be represented as a nontrivial monomial of weight $\alpha$ times a quasihomogeneous polynomial of weight zero. To make this representation unique, we will require that this quasihomogeneous polynomial is \emph{without zero roots},
$$
P_\alpha(t,\xi)=ct^j\xi^i\cdot \prod_{s=1}^k (\l_s-t^q\xi^p),\qquad c\prod_s\l_s\ne 0,\quad \sigma_P(\l_s)=0,\ \wgt(i,j)=\alpha.
$$

\begin{Def}\label{def:char-poly}
The univariate polynomial $\sigma=\sigma_P$ introduced by the above construction, is called the \emph{characteristic polynomial} of the quasihomogeneous pseudopolynomial $P$. Its (nonzero) roots are called \emph{characteristic numbers}.
\end{Def}


\subsection{Graded algebra of pseudopolynomials}
Let as before $w\in\Q_+$ be a rational weight and $\wgt(\cdot)$ the corresponding weight. The algebra $\^\Cs$ is naturally \emph{graded} by this weight, i.e., represented as a countable direct sum,
\begin{equation}\label{graded}
  \^\Cs=\bigoplus_{\alpha\in\Q}\Cs_\alpha,\qquad \Cs_\alpha=\{P\in\^\Cs :\wgt P=\alpha\}.
\end{equation}
The index $\alpha$ effectively ranges over the set $\Z_+-w\Z_+\in\Q$ which is completely ordered (i.e., it is discrete, bounded from below and unbounded from above). All other terms $\Cs_\alpha$ are trivial. This grading agrees with the structure of algebra:
\begin{equation}\label{gr-alg}
  \Cs_\alpha\cdot\Cs_\beta\subseteq\Cs_{\alpha+\beta},\qquad \forall \alpha,\beta\in\Q.
\end{equation}
Consequently, any pseudopolynomial $P\in\^\Cs$ can be expanded as a series $P=\sum_{\alpha\in\Q}P_\alpha$, which is in general infinite but always has a well-defined $w$-\emph{leading term} $P_*$ of the minimal weight $\alpha_*=\min \wgt\big|_{\D_P}$. It is tempting to use the imperfect notation $P=P_*+\cdots$ to state the corresponding fact.


\subsection{Commutative factorization problem}\label{sec:comm-factorization}
We will focus on the following special form of the factorization problem for pseudopolynomials: given $P\in\^\Cs$ with a Newton polygon $\D=\D_P$ and an admissible decomposition $\D=\D'+\D''$ into the Minkowski sum, construct a factorization $P=QR$ with $\D_Q=\D'$, $\D_R=\D''$. Some additional assumptions will be required.

\begin{Ex}
Assume that $P$ is a Fuchsian pseudopolynomial with pairwise distinct characteristic numbers $\l_1,\dots,\l_d$. Then its Newton polygon decomposes as the Minkowski sum of $d$ identical copies of $[0,1]\times\R_+$, so one can expect that it factors as a product of $d$ linear pseudopolynomials of the form
$P(t,\xi)=c(t)\prod_{s=1}^d (\xi-\boldsymbol \l_s(t))$. This is indeed the case, as follows from the (formal) Implicit Function Theorem: if all roots of $p_0$ are simple, $p_0'(\l_s)\ne 0$, then each root can be expressed as $\xi_s=\boldsymbol\l_s(t)\in\C[[t]]$, $\boldsymbol\l_s(0)=\l_s$.
\end{Ex}

\subsubsection{Factorization in $\C[[x,y]]$}
The factorization problem for pseudopolynomials can be instantly reduced to that for formal series in two variables, as mentioned in Remark~\ref{rem:local}. The factorization problem for such objects is well known, see \cite{wall}. If the series were convergent, then this would be the problem of determining all irreducible branches of the germ of a planar analytic curve in $\{f(x,y)=0\}\subset(\C^2,0)$. The answer is determined by the (classical) Newton diagram of the germ $f$ which is the graph of a piecewise affine convex function $\chi_f\:\R^+\to\R^+$ which \emph{decreases} to zero at some point $\le d$, cf.~with Remark~\ref{rem:local}. Slopes of this function are negative; as before, $f$ is called \emph{single-slope}, if its Newton diagram consists of a single edge.

\begin{Thm}\label{thm:locA}
A formal series $f\in\C[[x,y]]$ admits factorization into single-slope series $f=f_1\cdots f_m$.  \qed
\end{Thm}

With a single-slope series $f$ one can associate the its leading part (a quasihomogeneous polynomial), the corresponding characteristic polynomial $\sigma=\sigma_f(\l)$ and its (nonzero) roots $\l_1,\dots,\l_d$, the \emph{characteristic numbers}, exactly as in \secref{sec:qhg}. The only difference is that the weights assigned to $x,y$ are both natural numbers. Obviously, $\sigma_{f_1f_2}=\sigma_{f_1}\sigma_{f_2}\in\C[\l]$.

\begin{Thm}\label{thm:locB}
Assume that the characteristic numbers of a single-slope series $f\in\C[[x,y]]$ form two disjoint groups so that $\sigma_f(\l)=\sigma_1(\l)\sigma_2(\l)$ and $\gcd(\sigma_1,\sigma_2)=1$.

Then $f$ admits factorization $f=f_1f_2$ so that $\sigma_{f_i}=\sigma_i$, $i=1,2$. \qed
\end{Thm}

\begin{Cor}
Any single-slope series can be factored out as a product of terms each having a single characteristic number, eventually with nontrivial multiplicity. \qed
\end{Cor}

These theorems immediately imply the following two factorization results for pseudopolynomials.

\begin{Thm}\label{thm:ppA}
Any pseudopolynomial $P\in\^\Cs$ admits factorization into single-slope terms $P=P_1\cdots P_m$.
\end{Thm}

\begin{Thm}\label{thm:ppB}
Assume that the characteristic numbers of a single-slope pseudopolynomial $P\in\^\Cs$ form two disjoint groups so that $\sigma_P$ factors as $\sigma_P(\l)=\sigma_1(\l)\sigma_2(\l)$ with $\gcd(\sigma_1,\sigma_2)=1$.

Then $P$ admits factorization $P=P_1P_2$ so that $\sigma_{P_i}=\sigma_i$, $i=1,2$.
\end{Thm}

\begin{Cor}
Any single-slope series can be factored out as a product of terms each having a single characteristic number, eventually with nontrivial multiplicity.
\end{Cor}

\begin{proof}[Proof by reduction to the local case]
For a pseudopolynomial $P(t,\xi)\in\^\Cs$ of degree $n$ denote $f(x,y)=y^n P(x,1/y)$. Then $f$ is a formal series in $x$ and a polynomial in $y$, that is, an element from $\C[[x,y]]$. Let $f=f_1f_2$ be the factorization of $f$ in the assumptions of Theorem~\ref{thm:locA} (Theorem~\ref{thm:locB} respectively). By the Weierstrass theorem (formal), one can assume that modulo an invertible series $f_i$ are polynomial in $y$ of degrees $n_1,n_2$ respectively, with $n_1+n_2=n$. The invertible series must also be polynomial in $y$ of degree $0$, that is, a formal series from $\C[[x]]$. Setting $P_i(t,\xi)=\xi^{n_i}f_i(t,1/\xi)$ gives the required factorization of $P$.
\end{proof}

\subsubsection{About the proofs}\begin{small}
The modern proof of Theorems~\ref{thm:ppA} and~~\ref{thm:ppB} relies on the desingularization, a sequence of rational monomial transformations which simplify the curve (or the formal series). These transformations (blow-ups) have the form
\begin{equation}\label{bup}
  (x,y)\longmapsto (x,y/x)\qquad\text{or}\qquad (x,y)\longmapsto (x/y,y)
\end{equation}
and act on the support of a series by an affine transformation which allows to extract factors of the form $x^p$ or $y^q$. For instance, if $\D_f$ has a single ``homogeneous edge'' connecting the vertices $(0,p)$ and $(p,0)$ and the corresponding characteristic numbers $\l_1,\dots,\l_p$ are pairwise different, then after a single blow-up one can refer to the implicit function theorem for the proof that $f$ admits factorization into terms corresponding to nonsingular branches,
\begin{equation*}
  f(x,y)=\prod_{i=1}^p (x-\boldsymbol \l_i(x)y),\qquad \boldsymbol \l_i\in\C[[x]],\quad \boldsymbol \l_i(0)=\l_i.
\end{equation*}
In case of multiple characteristic values one has to refer to the Weierstrass Preparation theorem instead of the implicit function theorem.

A single-slope series with the single edge connecting $(0,p)$ and $(q,0)$ requires several blow-ups whose number and types are determined by the Euclid algorithm for computation of $\gcd(p,q)$. A slightly more delicate considerations are required when $f$ has more than one slope, but the idea remains the same. \par\end{small}

\section{Homological equation and its solvability}

In this section we return to the algebra of pseudopolynomials $\^\Cs=\C[[t]][\xi]$ and attempt to construct factorization in this algebra directly, following Newton's ideas.

\subsection{Formal factorization}\label{sec:formfact}
Consider an admissible polygon $\D\in\R_+^2$ and the weight function $\wgt=\wgt_w\:\Z^2\to\Q$ associated with a rational weight $w\in\Q_+$. Denote
\begin{equation}\label{qhn}
  \Cs_\alpha(\D)=\Cs_\alpha\cap\{\supp P\subseteq\D\}.
\end{equation}

Then the property \eqref{gr-alg} can be refined as follows: for any two admissible polygons $\D',\D''\subseteq\R^2_+$ and any $\alpha,\beta\in\Q_+$,
\begin{equation}\label{gr-alg-D}
  \Cs_\alpha(\D')\cdot\Cs_\beta(\D'')\subseteq\Cs_{\alpha+_\beta}(\D'+\D'').
\end{equation}

Let $P\in\^\Cs$ be a pseudopolynomial expanded into $w$-quasihomogeneous terms as $P=\sum_\gamma P_\gamma$, and assume that $\D=\D_P=\D'+\D''$ is the admissible decomposition of its Newton polygon. The factorization under the form $P=QR$ can be achieved by two formal expansions $Q=\sum_\alpha Q_\alpha$, $R=\sum_\beta R_\beta$, if and only if
\begin{equation}\label{factor-comm}
  P_\gamma=\sum_{\alpha+\beta=\gamma}Q_\alpha R_\beta,\qquad Q_\alpha\in\Cs_\alpha(\D'),\ R_\beta\in\Cs_\beta(\D'').
\end{equation}
Denote the leading terms of the three pseudopolynomials by $P_*,Q_*,R_*$ respectively (of weights $\gamma_*=\min \wgt\big|_\D$, $\alpha_*=\min \wgt\big|_{\D'}$, $\beta_*=\min \wgt\big|_{\D''}$) and assume that
\begin{equation}\label{factor-seed}
  P_*=Q_*R_*\in\Cs_{\gamma_*}(\D).
\end{equation}

Then \eqref{factor-comm} becomes  is an infinite \emph{triangular} system of \emph{linear algebraic equations} with respect to the unknown terms $Q_\alpha,R_\beta$ from the corresponding finite-dimensional linear spaces $\Cs_\alpha(\D')$, $R_\beta\in\Cs_\beta(\D'')$.

Indeed, each equation can be rewritten as
\begin{equation}\label{homolog}
  Q^*R_{\gamma-\alpha_*}+Q_{\gamma-\beta_*}R^*=-P_\gamma+\sum_{\alpha>\alpha_*,\ \beta>\beta_*}Q_\alpha R_\beta.
\end{equation}
The condition on the weights in the right hand side means that it involves only the terms of the weights strictly less than $Q_{\gamma-\beta_*}$ (resp., $R_{\gamma-\alpha_*})$. If these terms were already determined recursively from the equations \eqref{homolog} solved for all smaller values of $\gamma$, then the right hand side is known and we can study its solvability the equation in the weight $\gamma$ as well. The equation \eqref{factor-seed} serves as a base for this inductive process.

Solvability of these equations depends on the following data: the weight $w$, the two admissible polygons $\D',\D''$ and the initial quasihomogeneous polynomials $Q_*,R_*$ of the appropriate weights. Denote by $\H$ the linear operator (more precisely, a family (sequence) of linear operators $\H_\gamma$, $\gamma\in\Q_+$)
\begin{equation}\label{hom-op}
  \H\:\Cs_{\gamma-\alpha_*}(\D'')\times \Cs_{\gamma-\beta_*}(\D')\to\Cs_\gamma(\D'+\D''),\qquad (U,V)\mapsto Q_*U+R_*V.
\end{equation}

\begin{Def}\label{def:homolog}
The equation(s) $\H(U,V)=W$ is called the \emph{homological equation} associated with the data $\mathscr H=(w,\D',\D'',Q_*,R_*)$. The homological equation is called \emph{solvable}, if each operator $\H_\gamma$ is \emph{surjective} for all $\gamma\ge\gamma_*=\alpha_*+\beta_*$.
\end{Def}

This equation can be considered as the linearization of the nonlinear equation $P=QR$ at the ``point'' $Q_*,R_*$ in the same way as it appears in the theory of local normal forms of vector fields etc., see \cite{thebook}*{\parasymbol 4}. Its solvability very strongly depends on the corresponding data $\mathscr H$, in particular, on the choice of the seed polynomials $Q_*,R_*$.

\subsection{Examples}
We start with the extreme case where $w=0$. It implies that the ``Fuchsian'' part of a pseudopolynomial can be always factored out.

\begin{Ex}\label{ex:horizontal}
Let $w=0$. Then $\wgt=\deg_t$, and the the quasihomogeneous components are of the form $P_j=t^j p_j(\xi)$, $j=0,1,2,\dots$. Let $d=\deg p_0<n=\deg P$. Since the Newton diagram of $P$ contains a nontrivial horizontal segment of length $d<n=\deg P$, then  $\D_P=\D'+\D''$, where $\D'$ is a vertical semistrip $[0,d]\times\R_+$. In other words, the gap function $\chi_{\D'}$ vanishes identically on $[0,d]$, and $\chi_{\D''}$ is strictly positive on $(0,n-d]$, i.e., the corresponding Newton diagram has only nonzero slopes. Let $Q_*=P_*=p_0(\xi)$, $R_*=1$. Substituting the expansions
$$
 Q=p_0(\xi)+\sum_{j=1}^\infty t^j q_j(\xi), \deg q_j\le d,\quad R=1+\sum_{j=1}^\infty t^j r_j(\xi),\ \deg r_j\le n-d
$$
into \eqref{homolog}, we obtain an infinite series of identities in $\C[\xi]$,
\begin{equation}\label{fuchs-out}
\begin{aligned}
 p_0&=p_0,\\
 p_1&=q_1+r_1p_0,\\
 p_2&=q_2+r_1q_1+r_2p_0,\\
 \dots&{\makebox[0.3\columnwidth]{\dotfill}}\\
 p_j&=q_j+r_1q_{j-1}+\cdots+r_j p_0,
\end{aligned}
\end{equation}
The initial identity is trivially satisfied. The requirements that the support of $Q_\beta$ belongs to $\D'$ means that $\deg q_j\le d$ (then the second requirement will be automatically satisfied).

The homological equation \eqref{fuchs-out} can be inductively solved with respect to $q_j,r_j$ by the division with remainder of the polynomial $p_j-\sum_{k=1}^{j-1}r_kq_{j-k}$ by $p_0$. The remainder term $q_j$ can be guaranteed to be of degree $\le d-1$ (and then it will be uniquely determined), while $r_j$ will be the respective incomplete ratio. This gives a direct proof of a particular case of Theorem~\ref{thm:ppA}.
\end{Ex}

\subsection{Sylvester map}
As was explained in \secref{sec:qhg}, quasihomogeneous polynomials can be expressed as univariate polynomials in the basic monomial of weight zero. An analog of the homological equation \eqref{homolog} for univariate polynomials looks as follows.

Denote by $\C_n[\l]$ the linear space of polynomials of degree $\le n-1$, so that $\dim_\C\C_n[\l]=n$, and assume $q_*\in\C_n[\l]$, $r_*\in\C_m[\l]$. Then there is a linear map, called the \emph{Sylvester map},
\begin{equation}\label{sylv}
  \boldsymbol S\:\C_m[\l]\times\C_n[\l]\to \C_{m+n}[\l], \qquad (u,v)\longmapsto q_*u+r_*v
\end{equation}
(the matrix of this map in the natural basis is the Sylvester matrix of the two polynomials $p,q$). It is well known that the Sylvester map is bijective if and only if $\gcd(p,q)=1$.

However, it is very difficult to apply this result to study the homological equation \eqref{homolog}: the dimensions $\dim\Cs_\alpha(\D)$ depend on the weight $\alpha$ in a rather irregular way. In general, the homological operator $\H_\gamma$ acts between spaces of different dimensions. Thus proving directly its surjectivity is problematic. However, it follows indirectly from the established in \secref{sec:comm-factorization} results on factorization of pseudopolynomials.

\subsection{Solvability of the homological equation}\label{sec:solvability}
Let $\mathscr H=(w,\D',\D'',Q_*,R_*)$ be the data defining the homological operator $\H$.

Consider first the case where one of the polygons, say, $\D'$ is single-slope, $\PS(\D')=\{\rho\}$, and choose the weight $w=\rho$. Then $\alpha_*=0$, and $Q_*$ is a quasihomogeneous polynomial of weight zero with nonzero characteristic roots. If $\rho\notin \PS(\D'')$, then $\beta_*<0$, the weight achieves its minimum at a corner point and the leading term $R_*$ is a (nontrivial) monomial.

\begin{Thm}\label{thm:homA}
If $w=\rho\notin \PS(\D'')$, i.e., the polygons $\D',\D''$ have no common slope, then all homological operators are surjective and the homological equation is solvable in any weight $\gamma$.
\end{Thm}

The second case deals with factorization of the quasihomgeneous polynomials into terms of lower weight. Assume that $\PS(\D')=\PS(\D'')=\{\rho\}$ and the weight is chosen accordingly, $w=\rho$. Then $\alpha_*=\beta_*=0$, and the corresponding characteristic $\sigma'=\sigma_{Q_*}$ and $\sigma''=\sigma_{R_*}$ are defined as in Definition~\ref{def:char-poly}.

\begin{Thm}\label{thm:homB}
If $\gcd(\sigma',\sigma'')=1$, i.e., the two characteristic polynomials have no common roots, then all homological operators are surjective and the homological equation is solvable in any weight $\gamma$.
\end{Thm}

\begin{proof}[Proof of both theorems]
Consider a pseudopolynomial $P$ with the leading part $P_*=Q_*R_*$ with $Q_*,R_*$ as in, say, Theorem~\ref{thm:homA}. By Theorem~\ref{thm:ppA}, $P$ admits factorization of the form $P=(Q_*+\cdots)(R_*+\dots)$ \emph{regardless} of the higher terms of $P$.

Substituting this factorization, we see that for each weight $\gamma>\gamma_*$ the homological equation \eqref{homolog} admits a solution for \emph{some} right hand side. Yet since the term $P_\gamma$ can be changed \emph{arbitrarily} without affecting reducibility, we conclude that the equation $\H_\gamma(U,V)=W$, see \eqref{hom-op}, is solvable for \emph{any} $W$.

In exactly the same way Theorem~\ref{thm:homB} follows from Theorem~\ref{thm:ppA}.
\end{proof}

\begin{Rem}\label{rem:convergence}
Theorems~\ref{thm:ppA} and~\ref{thm:ppB} describe factorization of the pseudopolynomials both in the formal context (as stated) and in the analytic context. Consider the commutative algebra $\Cs=\C[\xi]\otimes_\C\mathscr O(t)$, cf.~with \eqref{comm-alg}, where $\mathscr O(t)$ is the algebra of germs of analytic functions at $(\C,0)$ which can be identified with the algebra of convergent Taylor series, the corresponding objects are called \emph{analytic pseudopolynomials}. Then each analytic pseudopolynomial $P\in\Cs$ can be factored as a product if two analytic pseudopolynomials $Q,R\in\Cs$. Moreover, among the (many) formal solutions $(Q,R)$ constructed using the homological equations, one can always find a convergent solution.
\end{Rem}

\section{Weyl algebra and factorization of differential operators}

\subsection{Weyl algebra}
Motivated by the arguments from \secref{sec:operators} on various representations of linear ordinary differential operators, we introduce the (formal) Weyl algebra $\W$ as the algebra of formal series\footnote{The classical Weyl algebra is \emph{generated} by two symbols with the same commutation relation, so consists of noncommutative \emph{polynomials} in these variables.} in the two non-commutative variables $t,\eu$ related by the commutation identity \eqref{comm}, which are in fact \emph{polynomials} in $\eu$.

Using the commutation rule, any element $L\in\W$ can be reduced to the infinite formal sum
\begin{equation}\label{formal-oper}
  L=L(t,\eu)=\sum_{(i,j)\in S}c_{ij} t^j\eu^i,\qquad S=\supp L\subset [0,\dots,n]\times\Z_+,\ c_{ij}\in\C\ssm\{0\},
\end{equation}
where all powers of $t$ always occur to the left from powers of $\eu$ (the canonical representation). The integer $n=\ord L$ is the order of the operator $L$, and $S$ is called its support, $S=\supp (L)$. The Newton diagram $\D_L$ is obtained from the support in exactly the same way as in the commutative case (convex hull and invariance by translations).

Because of the non-commutativity of $\W$, in general $\supp (LM)\not\subseteq \supp (L)+\supp (M)$. However, the identity \eqref{comm} implies that
\begin{equation}\label{e-decrease}
  t^j\eu^i\cdot t^{j'}\eu^{i'}=t^{j+j'}\eu^{i+i'}+\sum_{k<i+i'}c_{kl}\,t^l\eu^k.
\end{equation}
This together with Proposition~\ref{prop:left} proves that
\begin{equation}\label{minksum-W}
  \forall L,M\in\W\qquad \D_{LM}=\D_L+\D_M
\end{equation}
(cf.~with Proposition~\ref{prop:minksum}).

\begin{Def}
For any $L\in\W$ with the canonical representation \eqref{formal-oper} the pseudopolynomial $P=P(t,\xi)=\sum_{\supp L}c_{ij}t^j\xi^i$ with the same coefficients $c_{ij}$ will be called\footnote{The classical notion of the symbol of a differential operator collects only the terms involving the highest order derivatives.} the \emph{pseudosymbol} of $L$ and denoted by $\PP_L$.

Conversely, for a pseudopolynomial $P=P(t,\xi)=\sum c_{ij}t^j\xi^i \in\^\Cs$ we will denote $P(t,\eu)\in\W$ the result of substitution of $\eu$ instead of $\xi$, $L=\sum_{i,j}c_{ij}t^j\eu^i$.
\end{Def}

Needless to warn that the pseudosymbol is by no means functorial: in general $\PP_{LM}\ne\PP_L\PP_M$ and $P(t,\eu)Q(t,\eu)\ne PQ(t,\eu)$.

The correspondence $\W\to\^\Cs$, $L\mapsto\PP_L$ allows to associate with operators from $\W$ all notions that were introduced for the pseudopolynomials. Thus we define Fuchsian operators, single-slope operators, the Poincar\'e spectrum e.a. Obviously, the pseudosymbols of Fuchsian operators become what they should be.

\subsection{Filtration of the Weyl algebra}
Let $w\in\Q_+$ be a rational weight and $\wgt_w(\cdot)$ the corresponding weight function. However (unfortunately) since $\PP_{LM}\ne \PP_L\PP_M$, we do not have the grading of $\W$ by different weights, only \emph{filtration}.

Recall that each grading of an algebra, in particular, the grading $\Cs(\D)=\bigoplus_\alpha\Cs_\alpha(\D)$, canonically defines a filtration by subspaces
$$
 \U_\alpha(\D)=\bigcup_{\gamma\ge \alpha}\Cs_\alpha(\D),\qquad \alpha,\gamma\in\Q.
$$
This filtration is monotone decreasing, $\U_\alpha(\D)\subseteq\U_\beta(\D)$ if $\alpha\ge\beta$, satisfies the condition $\U_\alpha(\D)\cdot\U_\beta(\D)\subseteq\U_{\alpha+\beta}(\D)$. Conversely, the grading can be restored from the filtration as follows,
\begin{equation}\label{filtograd}
\Cs_\alpha(\D)=\U_\alpha(\D)/\U_\alpha^+(\D), \qquad\text{where}\qquad \U_\alpha^+(\D)=\bigcup_{\gamma>\alpha}\U_\alpha(\D).
\end{equation}

\begin{Def}
Let $\alpha\in\Q_+$ be a rational number and $\D$ an admissible polygon. We define $\W_\alpha(\D)$ as the subspace
\begin{equation}\label{filter}
  \W_\alpha(\D)=\{L\in\W: \PP_L\in\U_\gamma(\D)\}.
\end{equation}
In other words, $\W_\alpha(\D)$ denotes the $\C$-space of operators from $\W$ whose pseudosymbol contains only terms of weight $\alpha$ and higher.
\end{Def}

By definition, $\W_\alpha(\D)\supseteq\W_\beta(\D)$ if $\alpha\ge\beta$, so the spaces $\W_\alpha(\D)$ form a decreasing filtration of $\W(\D)$. This filtration agrees with the composition in $\W$ in the sense that
\begin{equation}\label{filter}
  \W_\alpha(\D)\cdot\W_\beta(\D)\subseteq\W_{\alpha+\beta}(\D)\qquad \forall \alpha,\beta\in\Q,
\end{equation}
cf.~with \eqref{gr-alg}.

Indeed, after reducing the composition of operators $L,M$ of weights $\alpha,\beta$ respectively to the canonical representation where all powers of $\eu$ occur to the left from all powers of $\eu$, we affect only terms of of order strictly greater than $\alpha+\beta$, as follows from \eqref{e-decrease} (recall that the weight of $\eu$ is negative $-w$).

Recall that for any choice of the weight we used \emph{all} rational numbers for labeling in the graded algebra $\^\Cs=\bigoplus_{\alpha\in\Q}\Cs_\alpha$: the homogeneous spaces $\Cs_\alpha$ could be nonzero only for countably many values forming an arithmetic progression (depending on $w$). In the same way the decreasing filtration of $\W$ by $\W_\alpha$ has ``jumps'' only at these values.

\begin{Prop}
For any rational $\alpha\in\Q$,
\begin{equation}\label{quotient}
  \W_\alpha(\D)/\W^+_\alpha(\D)=\Cs_\alpha(\D),\qquad\text{where}\quad
  \W_\alpha^+(\D)=\bigcup_{\gamma>\alpha}\W_{\gamma}(\D).
\end{equation}
\end{Prop}

\begin{proof}
This follows from \eqref{filtograd} and the definition of the subspaces $\W^+\alpha(\D)$.
\end{proof}

\subsection{Factorization in the Weyl algebra}
Assume that $L\in\W$, $\ord L=n$ and $\D_L=\D'+\D''$. We look for conditions guaranteeing that $L$ can be decomposed as $L=MN$ with $M,N\in\W$ and $\D_M=\D'$, $\D_N=\D''$.

Choose a weight $w\in\Q$ and expand $L$ as the series $L=\sum_\gamma P_\gamma(t,\eu)$, where $P_\gamma$ are the corresponding quasihomogeneous components of the pseudopolynomial $P=\PP_L\in\^\Cs(\D)$.

We will look for the factorization in the form $L=MN$ defined by indeterminate pseudopolynomials $Q=\PP_M\in\Cs(\D')$, $R=\PP_N\in\Cs(\D'')$. by inductively constructing them and try to mimic the formal arguments from \secref{sec:formfact}. All notations will be kept as similar as possible in the commutative case. We assume that both $Q,R$ are expanded as sums of quasihomogeneous components $Q=\sum Q_\alpha$, $R=\sum R_\beta$.

The leading quasihomogeneous terms $Q_*,R_*$ of the minimal weights $\alpha_*,\beta_*$ respectively, must yield factorization of the leading term $P_*$. Fix them and consider the equation $\PP_L=\PP(MN)$. Since $\W$ is non-commutative, the right hand side is not equal to $\PP_M\PP_N$, but for any $\alpha,\beta$ from \eqref{e-decrease} it follows that
\begin{equation}\label{pssymb-weight}
  Q_\alpha(t,\eu)R_\beta(t,\eu)=(Q_\alpha R_\beta)(t,\eu)\bmod \W_\gamma,\qquad \gamma=\alpha+\beta,
\end{equation}
that is, after reducing the composition of operators to the canonical form, the result will have the same leading terms of order $\gamma=\alpha+\beta$ as if the algebra were commutative.

This means that the pseudopolynomials $Q_\alpha$, $R_\beta$ can be inductively defined from the infinite ``triangular'' system of equations of the form
\begin{equation}\label{triang-w}
  P_\gamma=\sum_{\alpha+\beta=\gamma}Q_\alpha R_\beta + S_\gamma, \qquad Q_\alpha\in\Cs_\alpha(\D'),\ R_\beta\in\Cs_\beta(\D''),
\end{equation}
cf.~with \eqref{factor-comm}, where $S_\gamma\in\Cs_\gamma(\D)$ is the collection of terms accumulated from re-expansion of terms $P_{\alpha'},Q_{\beta'}$ with $\alpha'+\beta'<\gamma$, which were already found by the induction hypothesis.

The equations \eqref{triang-w} are identical to the equations \eqref{factor-comm}, and their solvability depends only on the properties of $Q_*,R_*$ and the Newton diagrams $\D,\D''$ as described in \secref{sec:solvability}. In particular, Theorems~\ref{thm:homA} and~\ref{thm:homB} imply the following results.

\begin{Thm}\label{thm:operA}
If $\D_L=\D'+\D''$ and the admissible polygons $\D',\D''$ have no common slope, then any operator $L\in\W(\D)$ admits a formal decomposition $L=MN$ with $M\in\W(\D')$, $N\in\W(\D'')$.
\end{Thm}

\begin{Thm}\label{thm:operB}
If $\D$ is a single-slope admissible polygon, $L\in\W(\D)$ has a characteristic polynomial $\sigma=\sigma_L\in\C[\l]$, then for any factorization $\sigma=\sigma'\sigma''$ with mutually prime polynomials $\sigma',\sigma''$ one can find a formal factorization $L=MN$ by two single-slope operators such that $\sigma_M=\sigma'$, $\sigma_N=\sigma''$.
\end{Thm}

\begin{proof}[Proof of both Theorems]
Each equation in the infinite series \eqref{triang-w} is of the form \eqref{homolog} with the only difference being an extra term $S_\gamma$ coming from the preceding equation. Its solvability follows from the surjectivity of the corresponding homological operator associated with the data $\mathscr H=(w,\D',\D'',Q_*,R_*)$.
\end{proof}

As an immediate corollary to these two theorems, we have the following result on reducibility.

\begin{Def}
A differential operator $L\in\W$ is called \emph{monic}, if it has a single slope, and the corresponding characteristic polynomial $\sigma_L$ has a single root.
\end{Def}

\begin{Thm}\label{thm:main}
Any differential operator $L\in\W$ admits a decomposition into the non-commutative product of monic operators.
\end{Thm}

\subsection{Remark on the convergence}
It is absolutely imperative to stress that all results on factorization of the differential operators, unlike their counterparts on pseudopolynomials, are only formal (cf.~with Remark~\ref{rem:convergence}). Technically, the difference between the two theories can be attributed to the fact that the passage from grading to filtration results in the growth of the number of terms in the right hand side of the homological equation \eqref{triang-w} compared with \eqref{factor-comm}.

However, the issue of the divergence of formal transformations, diagonalizing (say, in the non-resonant case) irregular singularities was studied in detail, and geometric obstructions were identified as a Stokes matrices \cite{thebook}*{\parasymbol 20G}.

The ambitious goal beyond this paper and its precursor \cite{shira} is to identify analytic obstructions to the formal Weyl classification and formal factorization in a similar form as a suitable cocycle over a punctured neighborhood $(\C,0)$. However, this project is still in its rudimentary stage.


\begin{bibdiv}
\begin{biblist}

\bib{arnold}{book}{
   author={Arnold, V. I.},
   title={Huygens and Barrow, Newton and Hooke},
   note={Pioneers in mathematical analysis and catastrophe theory from
   evolvents to quasicrystals;
   Translated from the Russian by Eric J. F. Primrose},
   publisher={Birkh\"auser Verlag, Basel},
   date={1990},
   pages={118},
   isbn={3-7643-2383-3},
   review={\MR{1078625}},
   doi={10.1007/978-3-0348-9129-5},
}

\bib{brieskorn}{book}{
   author={Brieskorn, Egbert},
   author={Kn\"orrer, Horst},
   title={Plane algebraic curves},
   series={Modern Birkh\"auser Classics},
   note={Translated from the German original by John Stillwell;
   [2012] reprint of the 1986 edition},
   publisher={Birkh\"auser/Springer Basel AG, Basel},
   date={1986},
   pages={x+721},
   isbn={978-3-0348-0492-9},
   review={\MR{2975988}},
   doi={10.1007/978-3-0348-5097-1},
}

\bib{sqh}{article}{
   author={Greuel, Gert-Martin},
   author={Pfister, Gerhard},
   title={On moduli spaces of semiquasihomogeneous singularities},
   conference={
      title={Algebraic geometry and singularities},
      address={La R\'abida},
      date={1991},
   },
   book={
      series={Progr. Math.},
      volume={134},
      publisher={Birkh\"auser, Basel},
   },
   date={1996},
   pages={171--185},
   review={\MR{1395180}},
}


\bib{thebook}{book}{
   author={Ilyashenko, Yulij},
   author={Yakovenko, Sergei},
   title={Lectures on analytic differential equations},
   series={Graduate Studies in Mathematics},
   volume={86},
   publisher={American Mathematical Society, Providence, RI},
   date={2008},
   pages={xiv+625},
   isbn={978-0-8218-3667-5},
   review={\MR{2363178 (2009b:34001)}},
}

\bib{kamgarpour}{article}{
    author={Kamgarpour, Masoud},
    author={Wheatherhog, Samuel},
    title={A New Approach to Jordan Decomposition for Formal Differential Operators},
    journal={\texttt{ArXiv}},
    volume={1702.03608v1},
    year={2017},
    month={2},
    pages={1--12},
    note={Preprint published on February 13, 2017.},
}


\bib{malgrange}{article}{
   author={Malgrange, Bernard},
   title={Sur la r\'eduction formelle des \'equations diff\'erentielles \'a singularit\'es irr\'eguli\`eres},
   journal={Pr\'epublication de l'Inst. Fourier, Grenoble},
   date={1979},
   reprint={
    title={Singularit\'es irr\'eguli\`eres},
    series={Documents Math\'ematiques},
        volume={5},
        author={Deligne, Pierre},
        author={Malgrange, Bernard},
        author={Ramis, Jean-Pierre},
    publisher={Soci\'et\'e Math\'ematique de France},
    date={2007},
   isbn={978-2-85629-241-9},
   review={\MR{2387754}},
   pages={97--107},
}

}

\bib{leanne}{thesis}{
  author={Mezuman, Leanne},
  school={Weizmann Institute of Science},
  year={2017},
  title={Classification of non-Fuchsian linear differential equations},
  type={M.Sc.~thesis}
}


\bib{mero-flat}{article}{
   author={Novikov, Dmitry},
   author={Yakovenko, Sergei},
   title={Lectures on meromorphic flat connections},
   conference={
      title={Normal forms, bifurcations and finiteness problems in
      differential equations},
   },
   book={
      series={NATO Sci. Ser. II Math. Phys. Chem.},
      volume={137},
      publisher={Kluwer Acad. Publ., Dordrecht},
   },
   date={2004},
   pages={387--430},
   review={\MR{2085816 (2005f:34255)}},
}

\bib{ore}{article}{
   author={Ore, \Ore ystein},
   title={Theory of noncommutative polynomials},
   journal={Ann. of Math. (2)},
   volume={34},
   date={1933},
   number={3},
   pages={480--508},
   issn={0003-486X},
   review={\MR{1503119}},
   doi={10.2307/1968173},
}

\bib{vdp-sing}{book}{
   author={van der Put, Marius},
   author={Singer, Michael F.},
   title={Galois theory of linear differential equations},
   series={Grundlehren der Mathematischen Wissenschaften [Fundamental
   Principles of Mathematical Sciences]},
   volume={328},
   publisher={Springer-Verlag, Berlin},
   date={2003},
   pages={xviii+438},
   isbn={3-540-44228-6},
   review={\MR{1960772}},
   doi={10.1007/978-3-642-55750-7},
}

\bib{robba}{article}{
   author={Robba, P.},
   title={Lemmes de Hensel pour les op\'erateurs diff\'erentiels. Application \`a
   la r\'eduction formelle des \'equations diff\'erentielles},
   language={French},
   journal={Enseign. Math. (2)},
   volume={26},
   date={1980},
   number={3-4},
   pages={279--311 (1981)},
   issn={0013-8584},
   review={\MR{610528}},
}

\bib{shira}{article}{
   author={Tanny, Shira},
   author={Yakovenko, Sergei},
   title={On local Weyl equivalence of higher order Fuchsian equations},
   journal={Arnold Math. J.},
   volume={1},
   date={2015},
   number={2},
   pages={141--170},
   issn={2199-6792},
   review={\MR{3370063}},
   doi={10.1007/s40598-015-0014-6},
}

\bib{vain-tren}{book}{
   author={Vainberg, M. M.},
   author={Trenogin, V. A.},
   title={Theory of branching of solutions of non-linear equations},
   note={Translated from the Russian by Israel Program for Scientific
   Translations},
   publisher={Noordhoff International Publishing, Leyden},
   date={1974},
   pages={xxvi+485},
   review={\MR{0344960}},
}

\bib{wall}{book}{
   author={Wall, C. T. C.},
   title={Singular points of plane curves},
   series={London Mathematical Society Student Texts},
   volume={63},
   publisher={Cambridge University Press, Cambridge},
   date={2004},
   pages={xii+370},
   isbn={0-521-83904-1},
   isbn={0-521-54774-1},
   review={\MR{2107253}},
   doi={10.1017/CBO9780511617560},
}


\end{biblist}
\end{bibdiv}
\end{document}